\documentclass{amsart}
   \usepackage{amsmath}
   \usepackage{amsfonts}   % if you want the fonts
   \usepackage{amssymb}    % if you want extra symbols
   \usepackage{graphicx}
\begin{document}

\title[Metric of the Thompson-Stein Groups]%
{The Word Problem and the Metric for the Thompson-Stein Groups}
\author{Claire Wladis}
\address{Department of Mathematics,
BMCC/CUNY, 199 Chambers St., New York, NY 10007, U.S.A.}
\email{cwladis@gmail.com}
%\keywords{Thompson's~group, Thompson-Stein~group, piecewise-linear~homeomorphisms, tree-pair~diagrams, word~problem, normal~form, metric}
\subjclass{20F65} %geometric group theory
\thanks{The author would like to thank Sean Cleary for his
support and advice during the preparation of this article.  The
author acknowledges support from the CUNY Scholar Incentive
Award and would like to thank the Techniche Universit\"{a}t
Berlin for its hospitality during the writing of this paper.}

\makeatletter
\def\maxwidth{%
  \ifdim\Gin@nat@width>\linewidth
    \linewidth
  \else
    \Gin@nat@width
  \fi
} \makeatother

\newcommand {\lft} {\mbox{${\mathcal L}$}}
\newcommand {\rt} {\mbox{${\mathcal R}$}}
\newcommand {\m} {\mbox{${\mathcal M}$}}
\newcommand {\mi} {\mbox{${\mathcal M^{\it i}}$}}
\newcommand {\mj} {\mbox{${\mathcal M^{\it j}}$}}
\newcommand {\mpp} {\mbox{${\mathcal M^{\it n-1}}$}}
\newcommand {\mone} {\mbox{${\mathcal M}^1$}}
\newcommand {\mtwo} {\mbox{${\mathcal M}^2$}}
\newcommand {\mpminusone} {\mbox{${\mathcal M}^{\it n-2}$}}

\newcommand {\lftnot} {\mbox{${\mathcal L_{\emptyset}}$}}
\newcommand {\lftl} {\mbox{${\mathcal L_{\rm L}}$}}
\newcommand {\rtnot} {\mbox{${\mathcal R_{\emptyset}}$}}
\newcommand {\rtr} {\mbox{${\mathcal R_{\rm R}}$}}
\newcommand {\rtj} {\mbox{${\mathcal R_{\it j}}$}}
\newcommand {\rtjone} {\mbox{${\mathcal R_{\it j_1}}$}}
\newcommand {\rtjtwo} {\mbox{${\mathcal R_{\it j_2}}$}}
\newcommand {\minot} {\mbox{${\mathcal M^{\it i}_{\emptyset}}$}}
\newcommand {\mionenot} {\mbox{${\mathcal M^{\it i_1}_{\emptyset}}$}}
\newcommand {\mitwonot} {\mbox{${\mathcal M^{\it i_2}_{\emptyset}}$}}
\newcommand {\mij} {\mbox{${\mathcal M^{\it i}_{\it j}}$}}
\newcommand {\mione} {\mbox{${\mathcal M}^{\it i}_1$}}
\newcommand {\mionejone} {\mbox{${\mathcal M^{\it i_1}_{\it j_1}}$}}
\newcommand {\mitwojtwo} {\mbox{${\mathcal M^{\it i_2}_{\it j_2}}$}}
\newcommand {\rti} {\mbox{${\mathcal R_{\it i}}$}}
\newcommand {\rtk} {\mbox{${\mathcal R_{\it k}}$}}
\newcommand {\rtone} {\mbox{${\mathcal R}_1$}}
\newcommand {\riplusone} {\mbox{${\mathcal R}_{i+1}$}}
\newcommand {\rtp} {\mbox{${\mathcal R_{\it n-1}}$}}
\newcommand {\rtstar} {\mbox{${\mathcal R_*}$}}
\newcommand {\mjnot} {\mbox{${\mathcal M^{\it j}_{\emptyset}}$}}
\newcommand {\mknot} {\mbox{${\mathcal M^{\it k}_{\emptyset}}$}}
\newcommand {\mpnot} {\mbox{${\mathcal M^{\it n-1}_{\emptyset}}$}}
\newcommand {\mik} {\mbox{${\mathcal M^{\it i}_{\it k}}$}}
\newcommand {\mkl} {\mbox{${\mathcal M^{\it k}_{\it l}}$}}
\newcommand {\mpq} {\mbox{${\mathcal M^{\it n-1}_{\it q}}$}}
\newcommand {\mpi} {\mbox{${\mathcal M^{\it n-1}_{\it i}}$}}
\newcommand {\mpj} {\mbox{${\mathcal M^{\it n-1}_{\it j}}$}}
\newcommand {\mistar} {\mbox{${\mathcal M^{\it i}_*}$}}
\newcommand {\mpstar} {\mbox{${\mathcal M^{\it n-1}_*}$}}

\newcommand {\gen} {\mbox{$\{x_0,x_1,\dots,,x_{n-1}\}$}}

\newcommand{\x}[1]{\mbox{$x_{#1}$}}
\newcommand{\xinv}[1]{\mbox{$x_{#1}^{-1}$}}
\newcommand{\xpm}[1]{\mbox{$x_{#1}^{\pm 1}$}}

\newcommand{\Fnm}{\ensuremath{F(n_1,...,n_k)}}
\newcommand{\Fn}{\mbox{$F(n)$}}
\newcommand{\Fm}{\ensuremath{F(m)}}
\newcommand{\F}{\mbox{$F$}}
\newcommand{\nary}{$(n_1,...,n_k)$--ary}
\newcommand{\nset}{\ensuremath{\{n_1,...,n_k\}}}
\newcommand{\Ttree}{\ensuremath{(T_-,T_+)}}
\newcommand{\Stree}{\ensuremath{(S_-,S_+)}}

\renewcommand{\v}{\ensuremath{\mathbf{v}}}

\newtheorem{thm}{Theorem}[section]
\newtheorem{prop}{Proposition}[section]
\newtheorem{lem}{Lemma}[section]
\newtheorem{cor}{Corollary}[section]
\newtheorem{rmk}{Remark}[section]
\newtheorem{defn}{Definition}[section]
\newtheorem{nota}{Notation}[section]

\begin{abstract}
We consider the Thompson-Stein group \Fnm{} where\\
$n_1,...,n_k\in\{2,3,4,...\}$, $k\in\mathbb{N}$.
%specifically whenever $n_1-1|n_j-1$ for all $j\in\{2,...,k\}$.
We highlight several differences between the cases $k=1$ and $k>1$, including the fact that minimal tree-pair diagram representatives of elements may not be unique when $k>1$. We establish how to find minimal tree-pair diagram representatives of elements of \Fnm{}, and we prove several theorems describing the equivalence of trees and tree-pair diagrams.  We introduce a unique normal form for elements of \Fnm{} (with respect to the standard infinite generating set developed by Melanie Stein) which provides a solution to the word problem, and we give sharp upper and lower bounds on the metric with respect to the standard finite generating set, showing that in the case $k>1$, the metric is not quasi-isometric to the number of leaves or caret in the minimal tree-pair diagram, as is the case when $k=1$.
\end{abstract}

\maketitle \tableofcontents

\section{Introduction}\label{intro}
\par In this paper, we consider a collection of groups of the form \Fnm{} which are natural generalizations of Thompson's group $F$, introduced by R. Thompson in the early 1960s (see \cite{F}).  For $k=1$, the metric properties of these groups are already well-known; we begin here an investigation of the metric properties of these groups for $k>1$.  As the use of tree-pair diagram representatives has been essential in the proofs of metric properties for $k=1$, we begin by developing a theory which allows us to use tree-pair diagram representatives to represent elements of these groups for $k>1$; this reveals several key differences between the case $k=1$ and $k>1$ with respect to the minimality and equivalence of tree-pair diagrams.  Then, using our theory of tree-pair diagram representatives along with infinite and finite presentations developed for these groups using the methods of Melanie Stein \cite{F23}, we derive a unique normal form which provides a solution to the word problem.  This normal form then gives us the necessary technical framework to give sharp upper and lower bounds on the metric of these groups.  It is well-known that when $k=1$, the metric is quasi-isometric to the number of leaves in a minimal tree-pair diagram representative of an element; we show here that this is not the case when $k>1$.

Groups of the form \Fnm\ were first introduced by Brown (see
\cite{Brown}, \cite{Brown2}) and were first explored in depth
by Stein in \cite{F23}; Bieri and Strebel also explored these
groups in a set of unpublished notes in which they considered a
larger class of piecewise-linear homeomorphisms of the real
line.  Higman, Brown, Geoghegan, Brin, Squier, Guzm\'{a}n,
Bieri and Strebel have all explored generalized families of
Thompson's groups and piecewise-linear homeomorphisms of the
real line (see \cite{Higman}, \cite{Brown2}, \cite{Brown},
\cite{BrinAuto}, \cite{BrinSquier}, and appendix of \cite{F23}
for details).  We consider the groups \Fnm\ because they are,
in a sense, the most general class of groups of
piecewise-linear homeomorphisms for which tree-pair diagrams
can be used as natural representatives, and tree-pair diagrams
have proved useful in establishing many of the metric
properties of $F$ and $F(n)$.  So far little is known about the
properties of the groups \Fnm{}; in \cite {F23}, Stein explores
homological and simplicity properties of this class of groups,
showing that all of them are of type $FP_\infty$ and finitely
presented.  In this paper she also gave a technique for
determining the infinite and finite presentations for each of
these groups.

\begin{defn}[Thompson-Stein group \Fnm{}]\label{Thomgpdefn}
The Thompson-Stein group \Fnm{}, where
$n_1,...,n_k\in\{2,3,4,...\}$ and $k\in\mathbb{N}$, is the
group of piecewise-linear orientation-preserving homeomorphisms
of the closed unit interval with finitely-many breakpoints in
$\mathbb{Z}[\frac{1}{n_1\cdots n_k}]$ and slopes in the cyclic
multiplicative group $\langle n_1,...,n_k\rangle$ in each
linear piece.  Thompson's group $F$ is then $F(2)$.
\end{defn}

Thompson's groups and their generalizations as piecewise-linear homeomorphisms of the real line have been studied extensively because they have occurred naturally in several different fields and because they have interesting and complex group structures with unique properties. For example, Thompson's groups $T$ and $V$, each of which contain the group $F$, were initially of interest to mathematicians because they were the first known examples of infinite, simple, finitely-presented groups.  $F$ was the first example of a torsion-free infinite-dimensional $FP_\infty$ group.  $F$ has exponential growth, but a quadratic Dehn function, and it is suspected that $F$ may be nonamenable, even though it  has no free abelian subgroup (finding a proof of the amenability or nonamenability of $F$ has remained an important open question for decades).  For further background information about Thompson's groups $F$, $T$ and $V$, see \cite{intronotes}.

\section{Representing elements using tree-pair diagrams}
\par Tree-pair diagrams have been used extensively to represent elements of the groups \F{} and \Fn{}. This method of representation has resulted in an exact method for calculating geodesic length in the Cayley graph (see \cite {fordhamthesis}, \cite {fordhamgd}, \cite {length}), which has been used to explore a number of group properties (see for example \cite{combpropF} or \cite{BurrilloFpMetric}).  For a detailed description of how tree-pair diagrams can be used to represent elements of \Fn{}, see \cite{notMAC}.

\begin{defn}[carets, trees, tree-pair diagrams; parents and children]
An $n$--ary caret is a graph with $n+1$ vertices joined by $n$ edges: one vertex has degree $n$ (the parent) and the rest have degree 1 (the children). We say that an $n$--ary caret is of type $n$.  An \nary\ tree is a graph formed by joining any finite number of carets by identifying the child vertex of one caret with the parent vertex of another caret (where every caret has a type in \nset).  An \nary\ tree-pair diagram is an ordered pair of \nary\ trees with the same number of leaves.
 \end{defn}

\par  We can see an example of a $(2,3)$--ary tree in Figure \ref{23arytree}.

\begin{figure}
  \includegraphics[width=0.5in]{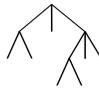}\\
  \caption{A example of a $(2,3)$--ary tree.}\label{23arytree}
\end{figure}

%% edit following paragraph%%%%%%%%%%%%%%%%%%%%%%%%%%%%%%%%%%%%%%%%%%%%%%%%%%%%%%%%%%%%%%%%%%%

\par We depict \nary\ trees so that child vertices are below parent vertices; the topmost caret in a tree is the {\it root caret} (or just the {\it root}) and its parent vertex is the {\it root vertex}. For any two vertices $a$ and $b$ in an tree, vertex $a$ is the {\it descendant} of vertex $b$ iff $b$ is on the directed path from the root node to vertex $a$. Vertices with degree 1 are {\it leaves}.

\subsection{Tree-pair diagrams as representatives of elements of \Fnm}

Every \nary\ tree represents a subdivision of $[0,1]$ in the
following way: Every vertex in the tree represents a closed
subinterval of $[0,1]$. The root vertex represents $[0,1]$. An
$n$--ary caret represents the subdivision of the parent vertex
interval $[a,b]$ into $n$ equally sized, consecutive, closed
subintervals. We then recursively assign a subinterval of
$[0,1]$ to every vertex in the tree.  Then the set of
subintervals corresponding to leaf vertices gives a subdivision
of $[0,1]$ into consecutive closed subintervals whose endpoints
occur exactly in $\mathbb{Z}[\frac{1}{n_1\cdots n_k}]$.

We number the leaves in a tree beginning with zero, in increasing order from left to right; a leaf's placement in this ordering is determined by the location of the subinterval, within the closed unit interval,  which that leaf represents.  If we let $\{[a_0,a_1],[a_1,a_2],...,[a_{m-1},a_m]\}$ and $\{[b_0,b_1],[b_1,b_2],...,[b_{m-1},b_m]\}$ represent the set of intervals represented by the leaves of $T$ and $S$ respectively, in increasing order, (where we note that $a_0,b_0=0$  and $a_m,b_m=1$), then we can turn the tree-pair diagram $(T,S)$ into the piecewise-linear orientation preserving homeomorphism of $[0,1]$:
\[f(x)=\cup_{i=1}^{m}f_i(x)\forall x\in[a_{i-1},a_i]\mbox{ where }f_i(x)=\frac{b_i-b_{i-1}}{a_i-a_{i-1}}(x-a_i)+b_i\]
Because by definition a tree has only finitely-many leaves, this homeomorphism must have only finitely-many breakpoints, and we note that the slopes $\frac{b_i-b_{i-1}}{a_i-a_{i-1}}$ will always be in $\langle n_1,...,n_k\rangle$.  So every \nary\ tree-pair diagram represents and element of \Fnm.

\begin{defn}[leaf valence, $\v(l_i)$]\label{leafpathval}
For a given leaf vertex $l_i$ in a tree, we will call the path from the root vertex to that leaf the leaf path of $l_i$; we will say that a specific caret is on a given leaf path if it has an edge on that path.  For any given $j\in\{1,...,k\}$, the $n_j$--valence of a leaf $l_i\in T$ is the number of $n_j$--ary carets which are on the leaf path of $l_i$; it is denoted by $v_{n_j}(l_i)$. If we refer to just the valence of a leaf-path, or $\v(l_i)$, this refers to the vector $\langle v_{n_1}(l_i),...,v_{n_k}(l_i)\rangle$.
\end{defn}

\begin{defn}[balanced tree]
A tree is balanced if $\v(l_i)=\v(l_j)\forall l_i,l_j\in T$.  As a consequence of Theorem \ref{rootthm}, we will see that a balanced tree can always be written as an equivalent tree containing rows of uniform caret type.
\end{defn}

\begin{thm}
 A map is an element of \Fnm{} iff it can be represented by a \nary\ tree-pair diagram.
\end{thm}
\begin{proof}
We have already shown that every \nary\ tree-pair diagram
represents an element of \Fnm{}. Now we prove the converse.
Because the elements of \Fnm{} are continuous piecewise-linear
maps with fixed endpoints, each element can be uniquely
determined by two sets containing the same number of interval
lengths.  So suppose we have an element $x\in\Fnm$ represented
by the sets
$[\frac{n_1}{d_1},...,\frac{n_l}{d_l}]\rightarrow[\frac{N_1}{D_1},...,\frac{N_l}{D_l}]$.
First we rewrite this so that all numerators are equal to one:
for any $n_i>1$, replace the submap
$\frac{n_i}{d_i}\rightarrow\frac{N_i}{D_i}$ with $n_i$--many
copies of the submap $\frac{1}{d_i}\rightarrow\frac{1}{D_i}$,
and renumber the indices of all interval lengths to adjust for
this new mapping.  Repeat this process for any $N_i>1$.  then
we will have
$x:[\frac{1}{d_1},...,\frac{1}{d_s}]\rightarrow[\frac{1}{D_1},...,\frac{1}{D_s}]$
for some $s\ge l$.

Now we create a tree-pair diagram for $x$ in the following
way:\\ let $M=LCM(d_1,...,d_s,D_1,...,D_s)$; we note that
$M\in\mathbb{Z}\left[\frac{1}{n_1\cdots n_k}\right]$.  Then we
create a tree-pair diagram for the identity consisting of a
pair of equivalent balanced trees with $M$--many leaves, each
representing a subinterval of length $\frac{1}{M}$.  Then for
arbitrary $i$, $\exists m_i,M_i\in\langle n_1,...,n_k\rangle$
such that $\frac{1}{M}\cdot m_i=\frac{1}{n_i}$ and
$\frac{1}{M}\cdot M_i=\frac{1}{N_i}$.  We begin with the map
$\frac{1}{d_1}\rightarrow\frac{1}{D_1}$; we need to add carets
to the tree-pair diagram so that the first $m_1$ leaves in the
domain tree map to the first $M_1$ leaves in the range tree. We
add a balanced tree containing $M_1$ leaves to each of the
first $m_1$ leaves in the domain tree and likewise a balanced
tree containing $m_1$ leaves to each of the first $M_1$ leaves
in the range tree; so the resulting $m_1M_1$--many leaves in
the domain tree, which represent the interval of length
$\frac{1}{d_1}$,  are mapped to the first $m_1M_1$ leaves in
the range tree, which represent the interval of length
$\frac{1}{D_1}$.  Continuing this process for all $i$ yields a
tree-pair diagram representative for $x$.
\end{proof}

\par The tree-pair diagram representative constructed in the proof above will typically not be minimal; however, we use this construction method in our proof rather than a more optimal one for the sake of brevity.  We note that while every \nary\ tree-pair diagram represents a unique element of \Fnm{}, there will always be an infinite number of tree-pair diagram representatives for any given group element.  This leads us to consider when trees and tree-pair diagrams might be equivalent or minimal.

\subsection{Equivalence of trees: Basic results}

\begin{defn}[equivalent trees]\label{equivtrees}Two \nary  trees are
equivalent if they represent the same subdivision of the unit interval.
\end{defn}

\begin{nota}[$L(T)$, $L(T_-,T_+)$, $L(x)$]
We will use the notation $L(T)$, $L(T_-,T_+)$, and $L(x)$ to denote the number of leaves in the tree $T$, in either tree of the tree-pair diagram $(T_-,T_+)$, and in either tree of the minimal tree-pair diagram representative for $x$ respectively.
\end{nota}

\par The number of leaves in a given tree-pair diagram refers to the number of leaves in either tree.  For example, each tree-pair diagram given in Figure \ref{equivtreessimple} has 8 leaves.

\begin{thm}\label{leafpaththm}
The \nary  tree $T$ is equivalent to the \nary  tree $S$ iff $L(T)=L(S)$ and $\v(l_i)=\v(k_i)$ for all leaves $l_i$ in $T$ and $k_i$ in $S$.
\end{thm}

\begin{proof}
The ``only if" statement contained in this theorem follows immediately from the definition of tree equivalence.  Now we prove the ``if" statement. We begin by considering two trees $T$ and $S$ and we suppose that $L(T)=L(S)$ and that $\v(l_i)=\v(k_i)$ for all leaves $l_i\in T$ and $k_i\in S$. If we let $\mathcal{L}(I)$ denote the length of the interval $I$, then for the intervals $I_i$ and $J_i$ represented by the leaves $l_i\in T$ and $k_i\in S$ respectively:
\[\mathcal{L}(I_i)=n_1^{-v_{n_1}(l_i)}\cdots n_k^{-v_{n_k}(l_i)}=n_1^{-v_{n_1}(k_i)}\cdots n_k^{-v_{n_k}(k_i)}=\mathcal{L}(J_i)\]
for all $i$.  Since $I_0=[0,a]$ for some $a\in[0,1]$ and $J_0=[0,b]$ for some $b\in[0,1]$, and since $\mathcal{L}(I_0)=\mathcal{L}(J_0)$, we must have $a=b$. If we have $I_n=[a_n,b_n]$ for $a_n,b_n\in[0,1]$ and $J_n=[a_n,c_n]$ for $a_n,c_n\in[0,1]$, then since $\mathcal{L}(I_n)=\mathcal{L}(J_n)$, we must have $b_n=c_n$, which, if we let $a_{n+1}=b_n=c_n$, implies that $I_{n+1}=[a_{n+1},b_{n+1}]$ for $a_{n+1},b_{n+1}\in[0,1]$ and $J_{n+1}=[a_{n+1},c_{n+1}]$ for $a_{n+1},c_{n+1}\in[0,1]$.  So by induction, we will have $I_i=J_i$ for all i.
\end{proof}

\begin{cor}\label{switchingcarets}
The trees $T$ and $S$ are equivalent iff $T$ can be obtained from $S$ by rearranging the order of carets on a given leaf path (perhaps for multiple leaf paths in the tree).
\end{cor}

\par We now proceed to discuss how to choose minimal tree-pair diagrams and how to compose tree-pair diagrams, which will be necessary in order to obtain criteria for determining when two tree-pair diagrams are equivalent.

\subsection{Minimal tree-pair diagrams}

\begin{defn}[equivalent tree-pair diagrams]\label{equivtpds}
Two \nary tree-pair diagrams are equivalent if they represent the same element of \Fnm{}.
\end{defn}

\begin{defn}[minimal tree-pair diagrams]\label{mintpds}
An \nary tree-pair diagram is {\it minimal} if it has the smallest number of leaves of any tree-pair diagram in the equivalence class of tree-pair diagrams representing a given element of \Fnm{}.
\end{defn}

We can reduce a tree-pair diagram by replacing it with an
equivalent tree-pair diagram with fewer leaves.  One way to
reduce a tree-pair diagram is to remove exposed caret pairs; an
exposed caret pair is a pair of carets of the same type, one in
each tree, such that all the child vertices of each caret are
leaves, and both sets of leaves have identical leaf index
numbers. Exposed caret pairs can be canceled because the
removal of an exposed caret pair does not change the underlying
map.  Cancelation of exposed caret pairs also has a natural
opposite: we can add a pair of identical carets to a tree-pair
diagram to the leaf with the same index number in each tree
without changing the underlying homeomorphism.

\par It is already well known that for $k=1$, the minimal tree-pair diagram for any element of \Fnm{} can be obtained solely through a sequence consisting of cancelation of exposed caret pairs (see \cite{notMAC} for more details on tree-pair diagram representatives of \Fnm{}).  When $k=1$, there is also always a unique minimal tree-pair diagram representative for any given element of \Fnm{}.  However, neither of these properties holds when $k>1$.

\begin{rmk}[Differences Between Minimality in \Fm\ and
\Fnm]\label{minrmk} There are some key differences between
minimal tree-pair diagram representatives in \Fm\ and \Fnm:
\begin{enumerate}
 \item There exist nonminimal \nary tree-pair diagrams which contain no exposed caret pairs (see Figure \ref{F23iden}).
 \item To obtain a minimal tree-pair diagram from a given tree-pair diagram, it may be necessary to add caret pairs to the tree-pair diagram (see Figure \ref{addtogetmin}).
 \item The minimal tree-pair diagram representative of an element of \Fnm{} may not be unique (see Figure \ref{equivtreessimple}).
\end{enumerate}
\end{rmk}

\begin{figure}[htbp]
    \centering
    \includegraphics[width=.75in]{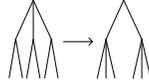}
    \caption{Nonminimal tree-pair diagram containing no exposed caret pairs.} \label{F23iden}
    \end{figure}

\begin{figure}[htbp]
\centering
\includegraphics[width=3.75in]{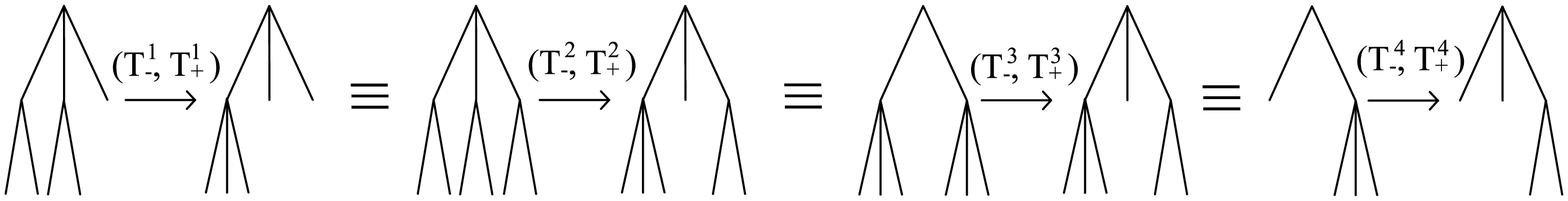}
\caption{$(T_-^1,T_+^1)$ must have carets added in order to obtain the minimal diagram $(T_-^4,T_+^4)$.} \label{addtogetmin}
\end{figure}

\begin{figure}[htbp]
    \centering
    \includegraphics[width=2.5in]{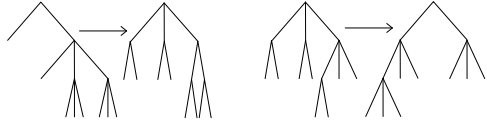}
    \caption{Two equivalent minimal tree-pair diagrams.}
    \label{equivtreessimple}
    \end{figure}

\par   We note that the two tree-pair diagrams in Figure \ref{equivtreessimple} do not even have equivalent domain trees or range trees; so there exist tree-pair diagrams containing no equivalent trees.

\subsection{Tree-pair diagram composition}

\par Composition of tree-pair diagrams is just composition of maps, where $xy$ denotes $x\circ y$.  So to find $xy$ for $x,y\in\Fnm{}$ with tree-pair diagrams $(T_-,T_+)$ and $(S_-,S_+)$ respectively, we need to make $S_+$ identical to $T_-$ (see Figure \ref{composition}). This can be accomplished by adding carets to $T_-$ and $S_+$ (and therefore to the leaves with the same index numbers in $T_+$ and $S_-$ respectively) until the valence of all leaves of both $T_-$ and $S_+$ are the same.  If we then let $T_-^*,T_+^*,S_-^*S_+^*$ denote $T_-,T_+,S_-,S_+$, respectively, after this addition of carets, then the product is $(S_-^*,T_+^*)$.

\begin{figure}[htbp]
\centering
\includegraphics[width=2.25in]{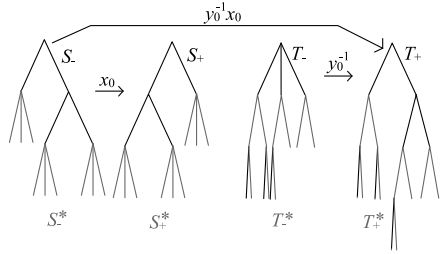}
\caption{Composition of $y_0^{-1},x_0\in F(2,3)$; black lines represent carets in the minimal tree-pair diagrams of  $y_0^{-1}$ and $x_0$ and grey lines represent carets added during composition.  } \label{composition}
\end{figure}

\par Since our operation is composition of maps, the tree-pair diagram representative for the right factor will always occur to the left during tree-pair diagram composition.  We note that it is not obvious that this process of adding carets to the domain tree of one element and the range tree of the other will necessarily terminate.  We proceed to prove this.

\begin{defn}[Common subdivision of two trees/tree-pair diagrams]
We say that a subtree $S$ of a tree $T$ is rooted iff the root
vertices of $S$ and $T$ are the same.  A common subdivision
tree of two trees $T$ and $S$ is an \nary tree $C_{T,S}$ such
that: $\exists$ trees $T^*,S^*$ with rooted subtrees $T_1,S_1$,
respectively, such that  $T^*\equiv S^*\equiv C_{T,S}$,
$T_1\equiv T$, and $S_1\equiv S$.
\end{defn}

\begin{thm}\label{commonsubdivision}
Any pair of \nary trees $T$ and $S$ has a common subdivision tree.
\end{thm}

\begin{proof}
We will let $\v(l_i\in T)=\langle\alpha_i^1,...,\alpha_i^k\rangle$, and we let $\v(L_j\in S)=\langle\beta_j^1,...,\beta_j^k\rangle$.  Let \[M_r=\displaystyle\max\Bigl\{\alpha_i^r,\beta_i^r\Bigm|i\in\bigl\{0,...,\max\{L(T),L(S)\}\bigr\}\Bigr\}\hbox{ for }r=1,...,k\]

  Now we need only add carets to $T$ and $S$ whose type is in \nset\  until the valence of  each leaf pair in $T$ and $S$ is identical.  For the sake of simplicity, we do this by adding carets until all valences in $T$  and $S$  are $\langle M_1,...,M_k\rangle$.  As long as this process terminates, it is clear that the resulting trees will be equivalent.

  Let $R$ be the balanced tree whose leaves have valence $\langle M_1,...,M_k\rangle$; then $L(R)=n_1^{M_1}\cdots n_k^{M_k}$, which is clearly finite. Using Theorem \ref{rootthm}, we can see that there must exist a tree $T^*\equiv R$ with a rooted subtree $T_1\equiv T$ and a tree $S^*\equiv R$ with a rooted subtree $S_1\equiv S$.
\end{proof}

\par In practice it will not typically be necessary to add carets until every leaf in $T^*$ and $S^*$ has
 valence $\langle M_1,...,M_k\rangle$.  For example, in Figure \ref{composition}, this process halted before all valences were equal.

\subsection{Equivalence of trees: Further results}

\begin{thm}[Equivalent trees]\label{equivtreesubstitution}
Two $(n_1,...,n_k)$--ary trees are equivalent iff one tree can be obtained from the other through a finite sequence of subtree substitutions of the type given in Figure \ref{equivtreesFnm} (for any $p,q\in\{n_1,...,n_k\}$ such that $p\ne q$).
\end{thm}

\begin{figure}[htbp]
\centering
\includegraphics[width=2in]{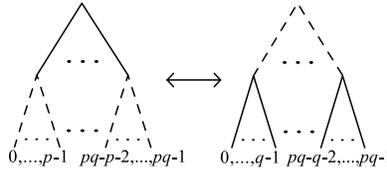}
\caption{Subtree substitutions of this form can be used to turn a tree into any equivalent tree (where $p,q\in\{n_1,...,n_k\}, p\ne q$).  Dotted line carets are $p$--ary and solid line carets are $q$--ary.} \label{equivtreesFnm}
\end{figure}

\par When $k=2$, $n_1=2$ and $n_2=3$,  the only substitution of this type is given in Figure \ref{F23iden}.

\par  Before we prove this theorem, we give another theorem, which will be used in the proof of the main theorem of this section.  This theorem will also be useful in its own right, in addition to its use in the proofs of other theorems of this paper.

\begin{thm}\label{rootthm}
An \nary tree $T$ can be written as an equivalent tree with an $m$--ary root caret (for some $m\in\{n_1,...,n_k\}$) iff $\v_m(l_i)>0$ for all $l_i\in T$. We note that this also holds for the root caret of subtrees within a larger tree.
\end{thm}

\begin{proof}
The ``only if" statement of this theorem is obvious, so we proceed to prove the ``if" statement.  We choose a tree $T$ with root caret of type $n\ne m$ such that $\v_m(l_i)>0$ for all $l_i\in T$.  We suppose that no equivalent tree exists which has a root of type $m$.

Then we consider the maximal rooted subtree $T^{max}$ of $T$ such that $\v_m(l_i)=0$ for all $l_i\in T^{max}$.  Since the root of $T$ is not of type $m$, $T^{max}$ will always be nonempty.  For example, consider the (2,3)--ary tree shown in Figure \ref{2aryleafformexbefore}.  In this case we let $m=2$.  Then $T^{max}$ consists of the grey hatched-line carets only.

\begin{figure}[htbp]
\centering
\includegraphics[width=1.25in]{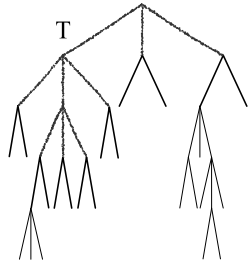}
\caption{A (2,3)--ary tree.} \label{2aryleafformexbefore}
\end{figure}

Because there must be at least one exposed caret in any tree, there must be at least one exposed caret in $T^{max}$.  Let $m_0$ denote the type of this caret (where we recall that $m_0\in\nset-\{m\}$). That same caret when viewed in $T$ will have $m_0$ many $m$--ary children (because $T^{max}$ is the maximal subtree such that $\v_m(l_i)=0$ for all $l_i\in T^{max}$. But the subtree consisting of that $m_0$--ary caret and its $m_0$ many $m$--ary children can be replaced in $T$ by the equivalent subtree consisting of a single $m$--ary caret with $m$ many $m_0$--ary children; through this substitution of equivalent subtrees, we obtain a tree $T_1$ which is equivalent to $T$.

We let $N(S)$ denote the number of carets in a tree $S$; now we
have $N\big(T_1^{max}\big)=N\big(T^{max}\big)-1$.  Now we
consider $T_1^{max}$; it also contains an exposed caret with
type $m_1$ in the set $\nset-\{m\}$, which has $m_1$ many
$m$--ary child carets in $T_1$, so we can replace this subtree
with its equivalent subtree with one $m$--ary root caret with
$m$--many $m_1$--ary children. If we continue this process so
that $T_i$ is the tree obtained by performing this kind of
substitution $i$--many times, then our inductive hypothesis is
that $N\big(T_i^{max}\big)=N\big(T^{max}\big)-i$ and $T_i\equiv
T$.  But this implies that
$N\big(T_{i+1}^{max}\big)=N\big(T^{max}\big)-(i+1)$ and
$T_{i+1}\equiv T$. So by induction,
$N\big(T_{N(T^{max})}^{max}\big)=N\big(T^{max}\big)-N\big(T^{max}\big)=0$
and $T_{N(T^{max})}\equiv T$. $N\big(T^{max}\big)$ is finite,
so $T_{N(T^{max})}$ is empty, which implies that the root caret
of $T_{N(T^{max})}$ is an $m$--ary caret, which contradicts our
initial assumption that there is no tree equivalent to $T$
which has an $m$--ary root caret.
\end{proof}

\begin{cor}\label{treeSTsub}
An \nary tree $T$ can be written as an equivalent tree $S$ with $m$--ary root caret iff $S$ can be transformed into $T$ through a finite sequence of subtree substitutions of the type given in Figure \ref{equivtreesFnm}.
\end{cor}

\par Now we proceed to prove the main theorem of this section.

\begin{proof}[of Theorem \ref{equivtreesubstitution}]
It is obvious that applying a sequence of subtree substitutions of the type given in Figure \ref{F23iden} to a given tree will always produce an equivalent tree.  Now we use induction and Corollary \ref{treeSTsub} to prove that any two equivalent trees can always be transformed into one another by a finite sequence  of these types of subtree substitutions.

\par Let $T$ and $S$ be a pair of equivalent \nary trees. We will transform $T$ into $S$ through a finite sequence of steps.  First we will make the root caret in $T$ identical to the root caret in $S$.  Then we will move down through $T$ level by level, and going from left to right through each level, we will change the caret type of each caret in $T$ until it is identical to that caret in that level of $S$.  In order to construct this formally, we define the following notation:  We let $D(T)$ denote the depth of the tree $T$, and we let $N_i(T)$ denote the number of carets in level $i$ of tree $T$ where we will number the levels in increasing order from top to bottom. Let $C_i(T,j)$ be the $j$th caret (counting in increasing order from left to right) in level $i$ of $T$ and let $Sub_i(T,j)$ be the subtree of $T$ whose root is $C_i(T,j)$ and which contains all carets which are descendants of $C_i(T,j)$ in $T$.

\par First we transform $T$ into to the tree $T^1$ by changing the caret type of $C_1(T,1)$ (the root caret of $T$) to the type of $C_1(S,1)$ (the root caret of $S$).  (If $T$ and $S$ already have the same root caret type, then $T^1=T$.)  By Corollary \ref{treeSTsub}, we can do this through a finite sequence of subtree substitutions of the type given in Figure \ref{F23iden}.  It is obvious that $N_1(T^1)=N_1(S)=1$ and that $C_1(T^1,j)$ and $C_1(S,j)$ have the same caret type for all possible $j$ (that is, for $j=1$).

\par Our inductive hypothesis is that $T^{n-1}\equiv S$, $N_{n-1}(T^{n-1})=N_{n-1}(S)$, and that $C_i(T^{n-1},j_i)$ and $C_i(S,j_i)$ will have the same caret type for all $i\le n-1$ and for all possible $j_i$.

\par Now we consider $Sub_n(T^{n-1},j)$ and $Sub_n(S,j)$ for $j=1,...,N_n(T^{n-1})=N_n(S)$.  For each $j$, $Sub_n(T^{n-1},j)$ can be transformed into the equivalent tree $Sub'_n(T^{n-1},j)$ which has the same root caret type as $Sub_n(S,j)$, by performing a finite sequence of substitutions of the type given in Figure \ref{F23iden} (by Corollary \ref{treeSTsub}).  Now we let $T^{n}$ be the tree equivalent to $T^{n-1}$ which is created by substituting $Sub'_n(T^{n-1},j)$ for $Sub_n(T^{n-1},j)$ in $T^{n-1}$ for each $j$.  Now $N_n(T^{n})=N_n(S)$ and $C_l(T^{n},j_l)$ and $C_l(S,j_l)$ have the same caret type for all $l\ge2$ and for all possible $j_l$.

\par So clearly $T^n\equiv S$, $N_n(T^n)=N_n(S)$, and
$C_i(T^n,j_i)$ and $C_i(S,j_i)$ have the same caret type for
all $i\le n$ and all possible $j_i$.  Therefore, by induction,
$T^{D(S)}\equiv S$ and $C_i(T^{D(S)},j_i)$ and $C_i(S,j_i)$
have the same caret type for all $i\le D(S)$ and for all
possible $j_i$, and therefore $T^{D(S)}\equiv S$ (where $D(S)$
is finite).
\end{proof}

\section{Equivalence of tree-pair diagrams}

\begin{thm}\label{equivtpdthm2}
 Any two equivalent \nary  tree-pair diagrams can be transformed into one another by a finite sequence consisting solely of the following two types of actions:
\begin{enumerate}
\item addition or cancelation of exposed caret pairs
\item subtree substitutions of the type given in Figure \ref{equivtreesFnm}
\end{enumerate}
\end{thm}
\begin{proof}
We begin this proof by first establishing a weaker formulation:

Two \nary  tree-pair diagrams are equivalent iff they can be transformed into one another by a finite sequence of the following two types of actions:
\begin{enumerate}
\item addition or cancelation of exposed caret pairs
\item substitution of equivalent subtrees
\end{enumerate}

The ``if" statement here is obvious, so we prove the ``only if" portion of this statement.  We know from Theorem \ref{commonsubdivision} that a common subdivision tree of $T_-$ and $S_-$ always exists which can be obtained by adding carets to $T_-$ and $S_-$; so we let $(T_-^*,T_+^*)$ and $(S_-^*,S_+^*)$ denote the tree-pair diagrams which are equivalent to $(T_-,T_+)$ and $(S_-,S_+)$ respectively, which are obtained by adding carets to $T_-$ and $S_+$ (and therefore by extension to $T_+$ and $S_-$) until the common subdivision tree $R$ of $T_-$ and $S_+$ is obtained.  So we have $R\equiv T_-^*\equiv S_-^*$, and $(T_-^*,T_+^*)\equiv(T_-,T_+)\equiv(S_-,S_+)\equiv(S_-^*,S_+^*)$, which implies that $(T_-^*,T_+^*)\equiv(S_-^*,S_+^*)$.  Therefore $T_-^*\equiv S_-^*$ implies that $T_+^*\equiv S_+^*$.  So we can obtain $(T_-,T_+)$ from $(S_-,S_+)$ by first adding the necessary carets to $(S_-,S_+)$ to obtain $(S_-^*,S_+^*)$, then substituting $T_-^*$ for $S_-^*$ and $T_+^*$ for $S_+^*$, and finally canceling the necessary carets in $(T_-^*,T_+^*)$ to obtain $(T_-,T_+)$.  Putting this result together with Theorem \ref{equivtreesubstitution}, the stated theorem immediately follows.
\end{proof}

\begin{cor}\label{twocarettypes}
If the domain and range trees in a tree-pair diagram have no
caret types in common, then the tree-pair diagram is minimal.
\end{cor}
\begin{proof}
No subtree substitution
will produce exposed caret pairs, because no exposed
carets in one tree will be of the same type as those in the other.  However, adding carets will not allow us to produce any
exposed caret pairs either (other than the caret pairs we add explicitly) because subtree substitution will only allow
us to move caret types from higher levels to lower levels, and the types of higher-level carets are distinct from those in the other tree.
\end{proof}

\section{Normal Form and Solution to the Word Problem}\label{normalform}
\par For the duration of this article we restrict our study of the group \Fnm{} to the case in which $n_1-1|n_j-1$ for all $j\in\{1,...,k\}$.  This is because groups which do not satisfy this criteria will have a significantly different group presentation.

\subsection{Infinite Group Presentation}
\par In \cite{F23}, Stein gave a method for computing the presentation of any group of the form \Fnm{} (whether or not $n_1-1|n_j-1$  for all $j\in\{1,...,k\}$); however, the only explicit presentation given in her paper for groups of this type was for $F(2,3)$.  So we now give an explicit infinite presentation for all groups \Fnm{} such that $n_1-1|n_j-1$ for all $j\in\{1,...,k\}$.

%\begin{defn}[right side]\label{rightside}
%The right side of a tree $T$ is the subtree of $T$ which
%consists of all carets which have an edge on the rightmost edge
%of the tree.  We let $R(T)$ denote the subtree of $T$ which
%consists solely of the right side of the tree.
%\end{defn}
%
%\begin{defn}[right caret]\label{right caret}
%A caret in a tree is a right caret if it has an edge on the rightmost edge of the tree (note that a root caret is a right caret using our definition, which is distinct from the usage of this term in \cite{notMAC}).
%\end{defn}

\begin{thm}[Infinite Presentation of Thompson's group $F(n_1,...,n_k)$]
Thompson's group \Fnm{}, where $n_1-1|n_i-1$ for all $i\in\{1,...,k\}$ has the following infinite presentation:\\
The generators of are:
\[Z=\{(y_{i})_0, (y_{i})_1, (y_{i})_2,...,(z_{j})_0, (z_{j})_1, (z_{j})_2,...\}\]
where $i\in\{2,...,k\},j\in\{1,...,k\}$.  The $(z_{j})_0,
(z_{j})_1, (z_{j})_2,...$ are the same as the standard
generators for $F(n_i)$ (see \cite{notMAC}).   We use $Z$ to
denote this generating set from now on.  The tree-pair diagram
representative of each generator is depicted in Figure
\ref{infgenFone}.

\begin{figure}[t]
\centering
\includegraphics[width=2.8in]{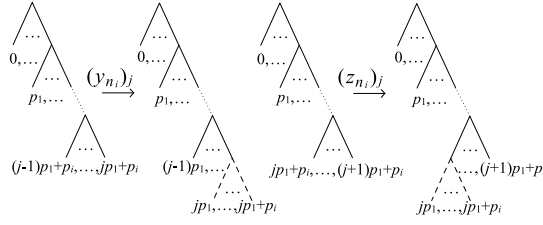}
\caption{The infinite generators of \Fnm{}, where $n_i=p_i+1$ for all $i\in\{1,...,k\}$.  Solid lines indicate $n_1$--ary carets and dashed lines indicate $n_i$--ary carets.} \label{infgenFone}
\end{figure}

The relations of this presentation are:
\begin{enumerate}
\item \label{conj} For all $r\in\{1,...,k\}$, $i<j$, and
    $\gamma_j$ a generator in\\
    $\left\{(y_{l})_j,(z_{i})_j|i\in\{1,...,k\},
    l\in\{2,...,k\}\right\}$,
    \[\gamma_j(z_{r})_i=(z_{r})_i\gamma_{j+(n_i-1)}\]
\item \label{equivtreeonright} For all $i,j\in\{1,...,k\}$, (where we will use the convention that $(y_1)_i$ is the identity),
    \begin{eqnarray*}(y_{i})_l(y_{j})_{l+1}(z_{j})_l(z_{j})_{l+n_j}(z_{j})_{l+2n_j}\cdots(z_{j})_{l+(n_i-2)n_j}=\\
    (y_{j})_l(y_{i})_{l+1}(z_{i})_l(z_{i})_{l+n_i}(z_{i})_{l+2n_i}\cdots(z_{i})_{l+(n_j-2)n_i}
    \end{eqnarray*}
\item \label{equivetreenotonright} For all $i,j\in\{1,...,k\}$,
    \begin{eqnarray*}(z_{i})_l(z_{j})_l(z_{j})_{l+n_j}(z_{j})_{l+2n_j}\cdots(z_{j})_{l+(n_i-1)n_j}=\\
    (z_{j})_l(z_{i})_l(z_{i})_{l+n_i}(z_{i})_{l+2n_i}\cdots(z_{i})_{l+(n_j-1)n_i}
    \end{eqnarray*}
\end{enumerate}
\end{thm}

\par The relations in \ref{conj} are the ``conjugation relations" that exist in $F$ (or more generally $F(n)$ for any $n\in\{2,3,4,...\}$).  Both the relations in \ref{equivtreeonright} and in \ref{equivetreenotonright} state that the presence of a subtree of the form of either of the two equivalent trees given in Figure \ref{equivtreesFnm} in an $(n_1,...,n_k)$--ary tree-pair diagram can be replaced with the other tree in the pair to obtain an equivalent tree-pair diagram; the relations in \ref{equivtreeonright} cover the case in which the carets on the right side of the subtree given in Figure \ref{equivtreesFnm} are on the right side of the larger tree-pair diagram, whereas the relations in \ref{equivetreenotonright} cover the case in which the carets on the right side of the subtree given in Figure \ref{equivtreesFnm} are not on the right side of the larger tree.

\par We will now introduce a normal form for elements of \Fnm{} in the standard infinite generating set $Z$ which is taken directly from the tree-pair diagram representative of the element.  This normal form will essentially give us a procedure for writing down an algebraic expression for any tree-pair diagram and for drawing a tree-pair diagram for any word in the normal form.  However, this normal form will only be as unique as the tree-pair diagram chosen to represent the element of \Fnm{}, so we begin by giving a method which will allow us to choose a unique minimal tree-pair diagram from among any set of minimal tree-pair diagram representatives.  This will produce a unique normal form that, while not quite as elegant and natural as in the case $k=1$, nonetheless provides a solution to the word problem.

\subsection{Unique Minimal Tree-pair Diagram Representatives}
\par In this section we give a set of conditions which will allow
us to chose a unique minimal tree-pair diagram representative
for any element of \Fnm.  In order to introduce criteria which
may be used to choose a unique minimal tree-pair diagram
representative, we introduce an ordering on the carets within a
tree.  Once we have a systematic way to order the carets in
tree, we can then produce an ordering on the tree-pair diagrams
in any set of minimal tree-pair diagrams, and once we have an
ordering on the set of equivalent minimal tree-pair diagrams,
then we can simply choose as the unique minimal tree-pair
diagram representative the first tree-pair diagram in the
ordering.

\begin{thm}\label{uniquealg}
We order all the carets in a tree by beginning at the top and
moving from left to right through each level as we work our way
down through to the bottom level of the tree.  Or, more
formally: we let $C_i(T,j)$ be the $j$th caret (counting in
increasing order from left to right) in level $i$ of the tree
$T$ and let $\tau_i(T,j)$ denote the type of the caret
$C_i(T,j)$. We then define the following order on the carets
within a tree: $C_i(T,j)<C_l(T,m)$ iff $i<l$, or $i=l$ and
$j<m$.  We let $\#C_i(T,j)$ denote the number of the caret
$C_i(T,j)$ in that order within the tree $T$.  Then the
following is a strict total order on the set of all minimal
tree-pair diagrams which
represent a given element of \Fnm:\\
\\
$\Ttree<\Stree$ iff one of the two conditions is met:
\begin{enumerate}
\item The first non-matching type in the domain trees is
    smaller in \Ttree\ than in \Stree: i.e. $\exists
    i_1,j_1$ such that
    $\tau_{i_1}(T_-,j_1)<\tau_{i_1}(S_-,j_1)$  and
    $\tau_i(T_-,j)=\tau_i(S_-,j)$ for all $i,j$ such that
    $\#C_i(T_-,j)<\#C_{i_1}(T_-,j_1)$ or
    $\#C_i(S_-,j)<\#C_{i_1}(S_-,j_1)$.
\item The domain trees of \Ttree\ and \Stree\ are
    identical, and the first non-matching type in the range
    trees is smaller in \Ttree\ than in \Stree: i.e.
    $\tau_i(T_-,j)=\tau_i(S_-,j)$ for all possible $i,j$
    and $\exists i_1,j_1$ such that
    $\tau_{i_1}(T_+,j_1)<\tau_{i_1}(S_+,j_1)$ and
    $\tau_i(T_+,j)=\tau_i(S_+,j)$ for all $i,j$ such that
    $\#C_i(T_+,j)<\#C_{i_1}(T_+,j_1)$ or
    $\#C_i(S_+,j)<\#C_{i_1}(S_+,j_1)$.
\end{enumerate}
\end{thm}
\begin{proof}
For any two tree-pair diagrams, if
$\tau_i(T_-,j)=\tau_i(S_-,j)$ and $\tau_i(T_+,j)=\tau_i(S_+,j)$
for all possible values of $i,j$, then every caret in the two
range trees will be identical, and every caret in the two
domain trees will be identical, and thus the two tree pair
diagrams will be identical.  So for two distinct tree-pair
diagrams, there must be at least one caret type in either the
range or domain tree which is not the same in both diagrams. So
it will always be possible to order any pair of minimal
tree-pair diagrams, and transitivity will hold.
\end{proof}

\begin{cor}\label{UNF}
A unique minimal tree-pair diagram representative can be chosen for any element $x$ of \Fnm.
\end{cor}
\begin{proof}
We choose as our unique representative that element of the set of minimal tree-pair diagram representatives of $x$ which comes first in the ordering described in Theorem \ref{uniquealg}.
\end{proof}

\subsection{The Normal Form}

\begin{thm}[Normal Form of Elements of \Fnm{}]\label{NFgen}
For any \nary\ tree-pair diagram $(T_-,T_+)$, the element $x\in\Fnm{}$ which \Ttree\ represents can be written in the form:

\[(y_{\delta_1})_{\beta_1}\cdots(y_{\delta_N})_{\beta_N}(z_{\epsilon_1})_{\alpha_1}^{e_1}\cdots(z_{\epsilon_M})_{\alpha_M}^{e_M}(z_{\theta_P})_{\Gamma_P}^{-d_P}\cdots(z_{\theta_1})_{\Gamma_1}^{-d_1}(y_{\theta_Q})_{\lambda_Q}^{-1}\cdots(y_{\theta_1})_{\lambda_1}^{-1}\]

where $\alpha_1\le\alpha_2\le\cdots\le\alpha_M\ne\Gamma_P\ge\Gamma_{P-1}\ge\cdots\ge\Gamma_1$, $\beta_i\le\beta_{i+1}-(n_{\delta_i}-1)\forall i\in\{1,...,N-1\}$, and $\lambda_{i+1}\le\lambda_{i}-(n_{\theta_{i+1}}-1)\forall i\in\{1,...,Q-1\}$,
where the conditions imposed on each of the $\alpha_i,\beta_i,\delta_i,\epsilon_i,\lambda_i,\theta_i,\phi_i,\Gamma_i,d_i,e_i$ depends only on the structure of the tree-pair diagram itself.  (We will describe these conditions in greater detail in Theorem \ref{NF}.)
\end{thm}

\par Any algebraic expression in this form immediately gives us
all the information necessary to write down a tree-pair diagram
representative for that word, and any \nary\ tree-pair diagram
will immediately give us enough information to write down a
word in this form which represented by the tree-pair diagram.
We will let $NF(x)$ denote this normal form for the word $x$,
derived using the unique minimal tree-pair diagram for $x$.

\par The remainder of this section will be devoted to proving our
normal form, and just as in the case \Fnm{} for $k=1$ in
\cite{BurrilloFpMetric}, the fact that any word can be factored
uniquely into a product of a positive word and its inverse will
be central to this proof. First we introduce the idea of
positive words.  A positive tree-pair diagram is a tree-pair
diagram whose domain tree contains only carets of type $n_1$
which are on the right side of the tree. A positive word in
\Fnm{} is a word whose normal form consists entirely of
generators with powers which are positive. Positive words are
precisely the words with positive tree-pair diagram
representatives.  (We note that there may be words which
consist entirely of positive powers of generators which are not
in the normal form, and that these words may not have positive
tree-pair diagram representatives.
 The word $(y_2)_0^2$ is a simple example of a word of this type.)
 In a positive tree-pair diagram, only the
range tree is nontrivial - as a result, for the duration of
this section, we will use a kind of shorthand notation, in
which we write $(*,T)$ to denote the positive tree-pair diagram
whose range tree is $T$.

Every \nary\ tree-pair diagram can be uniquely factored into a
product of two positive tree-pair diagrams.  If we have a
tree-pair diagram $(T_-,T_+)$ and we let $(S_-,S_+)$,
$(R_-,R_+)$ denote positive tree-pair diagrams where $S_+=T_-$
and $R_+=T_+$, then $(T_-,T_+)=(R_-,R_+)(S_-,S_+)^{-1}$.  From
this we can conclude that any element of \Fnm{} can be written
as the product of a positive word and the inverse of a positive
word.

So if a normal form exists for positive elements of \Fnm{}, then a
normal form exists for any element $w$ of \Fnm{}.  This normal
form can be obtained by writing
\[NF(w)=NF(w_+)NF(w_-^{-1})\]
where $w_+$ and $w_-$ are both positive words in \Fnm{} such
that $w_+w_-^{-1}=w$.

\subsection{The Normal Form of Positive Words}
For the duration of this section, we shall restrict ourselves
to describing the normal form for positive words only.  The
following definition will be central to this construction:

\begin{defn}[leaf exponent matrix]\label{matrix}
We will call a caret a right caret if it has an edge on the
right side of the tree.  Let $\eta_i$ be the minimal path from
the leaf $l_i\in T$ to the root vertex of $T$.  Now we consider
the subpath $\lambda_i$ of $\eta_i$ which begins at $l_i$ and
moves along $\eta_i$ only as far as possible by moving along
edges which are the leftmost edge of some non-right caret in
the tree.  If no movement is possible along $\eta_i$ which
satisfies this criteria, then we say that $\lambda_i$ is empty.

Now we break $\lambda_i$ down into subpaths
$\lambda_{i,1},...,\lambda_{i,r}$, which are chosen so that
each subpath is the maximal subpath containing one caret type.
We use the convention that $\lambda_{i,1}$ denotes the subpath
closest to the root, and $\lambda_{i,j}$ denotes the subpath
adjacent to $\lambda_{i,j-1}$, where $\lambda_{i,j}$ is closer
to the leaf $l_i$ than $\lambda_{i,j-1}$.  Then the leaf
exponent matrix for $l_i$ is
\[E_i=\left(\begin{array}{ccc}
                    \epsilon_{i,1} & \cdots & \epsilon_{i,m}\\
                    d_{i,1} & \cdots & d_{i,m}
                    \end{array}\right)\]
where $\epsilon_{i,j}\in\nset$ is the type of the carets on the
path $\lambda_{i,j}$ and $d_{i,j}$ is the number of carets on
the path $\lambda_{i,j}$.
\end{defn}

\par For example, some of the leaf exponent matrices for $T_-$ in Figure \ref{NFex} are: \\
$E_0=\left(\begin{array}{cc} 3 & 2\\ 1 & 1 \end{array}\right)$, $E_1=(\hbox{ })$, and $E_2=\left(\begin{array}{cc} 2 & 3\\ 1 & 2 \end{array}\right)$.

\begin{thm}[Normal Form of Positive Words]\label{NF}
We let $R(T)$ denote the subtree of the tree $T$ consisting
entirely of carets with an edge on the right side of the tree.
For any positive \nary\ tree-pair diagram $(T_-,T_+)$, the
positive word $w\in\Fnm{}$ which \Ttree\ represents can be
written in the form:

\[(y_{\delta_1})_{\beta_1}\cdots(y_{\delta_N})_{\beta_N}(z_{\epsilon_1})_{\alpha_1}^{e_1}\cdots(z_{\epsilon_M})_{\alpha_M}^{e_M}\]

where $\beta_{i+1}\ge(n_{\delta_i}-1)+\beta_i$,
$\alpha_{i+1}\ge\alpha_i$, and $\epsilon_i\ne\epsilon_{i+1}$,
$\forall i$.  Here is the procedure:
\begin{enumerate}
\item \label{ycond}
    $(y_{\delta_1})_{\beta_1}\cdots(y_{\delta_N})_{\beta_N}$:
    Consider $R(T_+)$; let $N$ be the number of
    non-$n_1$--ary carets in $R(T_+)$, and let
    $\wedge_{\beta_i}$ denote the non-$n_i$--ary caret with
    leftmost leaf index number $(n_1-1)\beta_i$ in $R(T_+)$
    with type $n_{\delta_i}$.
\item \label{zcond}
    $(z_{\epsilon_1})_{\alpha_1}^{e_1}\cdots(z_{\epsilon_M})_{\alpha_M}^{e_M}$:
    Consider all of the carets in $T_+$ which are not in
    $R(T_+)$; let $M$ be the number of carets in $T_+$
    which have non-empty leaf exponent matrices, let
    $l_{\alpha_i}$ denote the leaf with nonempty leaf
    exponent matrix $\left(\begin{array}{ccc}
                    \epsilon_{i} & \cdots & \epsilon_{i+m}\\
                    e_{i} & \cdots & e_{i+m}
                    \end{array}\right)$ and index number $\alpha_i$.
                    Then $\alpha_i=\alpha_{i+1}$ iff $l_{\alpha_i}$ has a leaf exponent matrix of the form $\left(\begin{array}{cccc}
                    \cdots & \epsilon_{i} & \epsilon_{i+1}& \cdots \\
                    \cdots & e_{i} & e_{i+1} & \cdots
                    \end{array}\right)$.
\end{enumerate}
When \Ttree\ is the unique minimal tree-pair diagram
representative of $w$, then this algebraic expression is the
unique normal form of $w$.
\end{thm}

\par The definition of the normal form given in this theorem
is somewhat technical - for a concrete example, see the example
given in subsection \ref{NFexsect}.

\begin{proof}
The proof of this theorem will proceed as follows:  First we
will show briefly that the normal form of a positive tree-pair
diagram can be factored uniquely into two parts:
\begin{enumerate}
\item The ``y" part, which represents
    $(y_{\delta_1})_{\beta_1}\cdots(y_{\delta_N})_{\beta_N}$
\item The ``z" part, which represents
    $(z_{\epsilon_1})_{\alpha_1}^{e_1}\cdots(z_{\epsilon_M})_{\alpha_M}^{e_M}$
\end{enumerate}
Then we will proceed to construct the normal form for each
part: first the section generated by the ``y" generators, and
then the section generated by the ``z" generators.  So we
proceed to factor $w$ uniquely as the product of two words $y$
and $z$.

Let $R^c(T)$ represent the tree which is derived from the tree
$T$ by replacing each $n_j$--ary right caret, for all $j\ne1$,
with a string of $\frac{n_j-1}{n_1-1}$ many right $n_1$--ary
carets, and let $R(T)$ represent the subtree of $T$ which
consists solely of carets on the the right side of the tree.
Then $(*,T)$ can be uniquely factored into a product of
$\left(*,R(T)\right)$ and $\left(*,R^c(T)\right)$:
\[(*,T)=\left(*,R(T_+)\right)\left(*,R^c(T_+)\right)\]

To see that this is true, we let $y=\left(*,R(T)\right)$ and
$z=\left(*,R^c(T)\right)$ for a given $w=(*,T)$, and we depict
the product $w=yz$ in Figure \ref{yzprod}.
\begin{figure}[t]
\centering
\includegraphics[width=2in]{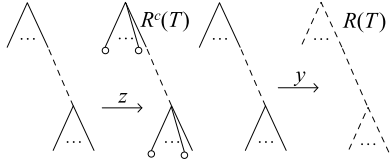}
\caption{The product $yz$ for positive word $w=yz$.  Solid lines indicate $n_1$--ary carets and dotted lines indicate $n_j$--ary carets for $j\in\{1,...,k\}$.} \label{yzprod}
\end{figure}
If we can show that
$NF(y)=(y_{\delta_1})_{\beta_1}\cdots(y_{\delta_N})_{\beta_N}$
and
$NF(z)=(z_{\epsilon_1})_{\alpha_1}^{e_1}\cdots(z_{\epsilon_M})_{\alpha_M}^{e_M}$,
then our theorem will follow immediately because the
factorization $w=yz$ is unique. Now we proceed to prove that
$NF(y)=(y_{\delta_1})_{\beta_1}\cdots(y_{\delta_N})_{\beta_N}$.

We show that
\[NF(y)=(y_{\delta_1})_{\beta_1}\cdots(y_{\delta_N})_{\beta_N}\]
where the caret with leftmost leaf index $(n_1-1)\beta_i$ in
$R(T)$ is type $n_{\delta_i}$ for $i=1,...,N$.

\par We prove this by induction.  Our inductive hypothesis is the statement itself.  It is clearly true when N=1.
Let
$y_l=(y_{\delta_i})_{\beta_i}\cdots(y_{\delta_l})_{\beta_l}$
and consider the product $y_l(y_j)_i$ for $i\ge L(y_l)-1$.  We
note that $L(y_l)=\beta_l+\eta_{\delta_l}$, so we have $i\ge
\beta_l+(\eta_{\delta_l}-1)$. We can see this composition in
Figure \ref{multbyy}
\begin{figure}[t]
\centering
\includegraphics[width=2in]{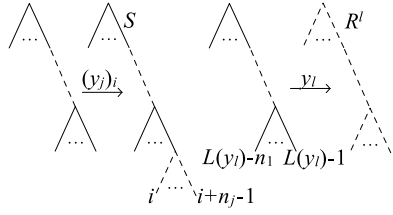}
\caption{The product $y_l(y_j)_i$.  Here $i\ge L(y_l)-1$, solid lines indicate $n_1$--ary carets, and dotted lines indicate $n_l$--ary carets for some $l\in\{2,...,l\}$ in $S_+$ and $l\in\{1,...,k\}$ in $T_+$ (and we assume that all dotted-line carets are not necessarily of the same type).} \label{multbyy}
\end{figure}
where we can see that the minimal tree-pair diagram
representative for $y_l(y_j)_i$ will be positive and will have
a range tree obtained from $R^l$ by the addition of (possibly
some $n_1$--ary right carets followed by) a single $n_j$--ary
right caret with leftmost leaf index number $i$.  So
$y_l(y_j)_i=(y_{\delta_1})_{\beta_1}\cdots(y_{\delta_l})_{\beta_l}(y_j)_i$
where $i>\beta_l$ for all $l$.

Now we show that
$NF(z)=(z_{\epsilon_1})_{\alpha_1}^{e_1}\cdots(z_{\epsilon_M})_{\alpha_M}^{e_M}$,
as described in Theorem \ref{NF}.  We show that
\[NF(z)=(z_{\epsilon_1})_{\alpha_1}^{e_1}\cdots(z_{\epsilon_M})_{\alpha_M}^{e_M}\hbox{ where }\alpha_{i+1}\ge\alpha_i\hbox{  and }\epsilon_i\ne\epsilon_{i+1}, \forall i\]

where $M$ is the number of carets in $R^c(T)$ with non-empty
leaf exponent matrices, the leaf $l_{\alpha_i}\in R^c(T)$ has
and non-empty leaf exponent matrix $\left(\begin{array}{ccc}
                    \epsilon_{i} & \cdots & \epsilon_{i+m}\\
                    e_{i} & \cdots & e_{i+m}
                    \end{array}\right)$.
    and $\alpha_i=\alpha_{i+1}$ iff $l_{\alpha_i}$ has leaf exponent matrix of the form $\left(\begin{array}{cccc}
                    \cdots & \epsilon_{i} & \epsilon_{i+1}& \cdots \\
                    \cdots & e_{i} & e_{i+1} & \cdots
                    \end{array}\right)$.

\par We prove this by induction, which will be motivated by the
following idea.  Let $x=(*,S)$ be a positive word.  Then the
product $x(z_j)_i$ has positive minimal tree-pair diagram
$(*,S^*)$, where $S^*$ is identical to $S$ except for an
$n_j$--ary caret hanging off the leaf that had index $i$ in $S$
(and possibly some $n_1$--ary right carets, whenever
$L(x)-n_1<\lfloor\frac{i}{n_1-1}\rfloor$).

\par To see this, we consult Figure \ref{multbygammaji}, which depicts the composition $x(z_j)_i$.
\begin{figure}[t]
\centering
\includegraphics[width=2in]{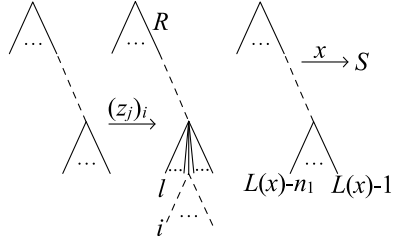}
\caption{The product $x(z_j)_i$ for positive word $w$.  Here $l=\lfloor\frac{i}{n_1-1}\rfloor$, solid lines indicate $n_1$--ary carets, and dotted lines indicate $n_j$--ary carets.} \label{multbygammaji}
\end{figure}
If $L(x)-n_1\ge l$, then to perform this composition, all we
need do is add an $n_j$--ary caret to the leaf with index $i$
in both trees of the tree-pair diagram for $x$.  If
$L(x)-n_1<l$, then to perform this composition, we will need to
add a string of $\frac{L(x)-n_1-l}{n_1-1}$ $n_1$--ary right
carets to the last leaf in both trees, and then we must add an
$n_j$--ary caret to the leaf with index $i$ in the resulting
tree-pair diagram representative for $x$.

\par Now suppose that the exponent matrices for leaves in $x$
with index greater than some fixed $I$ are empty.  Our
inductive hypothesis is that $x$ can be written in the form
given above.  Since the exponent matrix for leaves with index
greater than $I$ is empty, we have $\alpha_M\le I$.  We now
know that
$x(z_j)_I=(z_{\epsilon_1})_{\alpha_1}^{e_1}\cdots(z_{\epsilon_M})_{\alpha_M}^{e_M}(z_j)_I$
(where $\alpha_M\le I$) is represented by the positive minimal
tree-pair diagram representative with range tree obtained from
the range tree of
$x=(z_{\epsilon_1})_{\alpha_1}^{e_1}\cdots(z_{\epsilon_M})_{\alpha_M}^{e_M}$
by adding an $n_j$--ary caret at the leaf which had index
number $I$ in $S$ and (possibly) some $n_1$--ary carets added
to the right side of the tree.  Since adding $n_1$--ary carets
to the last leaf in the tree and adding a caret to the leaf
with index $I$ in $S$ does not change the leaf numbering of any
leaf with nonempty exponent matrix (since only leaves with
index numbers less than $I$ have nonempty exponent matrices),
it is clear that if we let $\epsilon_{M+1}=j$ and
$\alpha_{M+1}=I$, then
$x(z_j)_I=(z_{\epsilon_1})_{\alpha_1}^{e_1}\cdots(z_{\epsilon_M})_{\alpha_M}^{e_M}(z_{\epsilon_j})_{I}$
and the conditions of the proposition are satisfied.
\end{proof}

\subsubsection{Normal Form Example}\label{NFexsect} We find the normal form of the
element given in Figure \ref{NFex}.

\begin{figure}[htbp]
\centering
\includegraphics[width=2.5in]{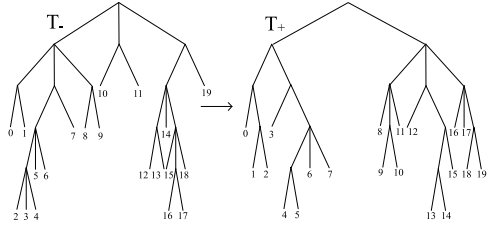}
\caption{Tree-pair diagram representative of an element $w$ of
$F(2,3)$} \label{NFex}
\end{figure}

\begin{eqnarray*}
NF(w) & = & (y_2)_1(y_2)_3(z_1)_0^2(z_1)_1(z_1)_3(z_2)_4(z_1)_4(z_2)_8(z_1)_9(z_1)_{12}(z_1)_{13}^2\\
&   & \cdot (z_1)_{16}^{-1}(z_2)_{15}^{-1}(z_1)_{12}^{-1}(z_2)_{12}^{-1}(z_1)_{10}^{-1}(z_1)_8^{-1}(z_2)_2^{-2}(z_1)_2^{-1}(z_1)_0^{-1}(z_2)_0^{-1}(y_2)_0^{-1}
\end{eqnarray*}

\par We now use the unique normal form to estimate the metric of \Fnm.

\section{The Metric on \Fnm{}}
\subsection{Standard Finite Presentation}
\par In \cite{F23} Stein also gave a method for finding finite presentations of \Fnm\ and gave a few examples.  However, she does not state these presentations explicitly, so we do so below.

\begin{thm}[Finite Presentation of Thompson's group $F(n_1,...,n_k)$]
Thompson's group \Fnm{}, where $n_1-1|n_i-1$ for all $i\in\{1,...,k\}$ has the following finite presentation:\\
The generators of this presentation are:
\[(y_{i})_0,..., (y_{i})_{n_i-1},(z_{j})_0,..., (z_{j})_{n_i-1}\]
where $i\in\{2,...,k\}$, $j\in\{1,...,k\}$.  We will use $X$ to denote this standard finite
generating set.

The relations of this presentation are:
\begin{enumerate}
\item \label{conj} For $(\gamma_m)_j$ a generator in the set
    $\{(y_{m})_j,(z_{m})_j|m\in\{1,...,k\}\}$, $j=0,...,n_m-1$
    \begin{enumerate}
    \item $(z_l)_i^{-1}(\gamma_m)_j(z_l)_i=(z_l)_0^{-1}(\gamma_m)_j(z_l)_0$
    \item $(z_l)_i^{-1}(z_l)_0^{-1}(\gamma_m)_j(z_l)_0(z_l)_i=(z_l)_0^{-2}(\gamma_m)_j(z_l)_0^2$
    \item $(z_l)_{n_l-1}^{-1}(z_l)_0^{-2}(\gamma_m)_1(z_l)_0^2(z_l)_{n_l-1}=(z_l)_0^{-3}(\gamma_m)_j(z_l)_0^3$
    \end{enumerate}
    for all $l\in\{1,...,k\}$ whenever $i<j$.
\item \label{equivtreeonright} For all $i,j\in\{1,...,k\}$, (where we use the convention that $(y_1)_i$ is the identity)
    \begin{enumerate}
    \item $(y_{i})_0(y_{j})_{1}(z_{j})_0(z_{j})_{n_j}\cdots(z_{j})_{(n_i-2)n_j}=
    (y_{j})_0(y_{i})_{1}(z_{i})_0(z_{i})_{n_i}\cdots(z_{i})_{(n_j-2)n_i}$
    \item
        $(y_{i})_1(y_{j})_{2}(z_{j})_1(z_{j})_{1+n_j}\cdots(z_{j})_{1+(n_i-2)n_j}\\=
        (y_j)_1(y_{i})_{2}(z_{i})_1(z_{i})_{1+n_i}\cdots(z_{i})_{1+(n_j-2)n_i}$
    \end{enumerate}
\item \label{equivetreenotonright} For all
    $i,j\in\{1,...,k\}$,
    \begin{enumerate}
    \item
    $(z_{i})_0(z_{j})_0(z_{j})_{n_j}\cdots(z_{j})_{(n_i-1)n_j}=
    (z_{j})_0(z_{i})_0(z_{i})_{n_i}\cdots(z_{i})_{(n_j-1)n_i}$
    \item
    $(z_{i})_1(z_{j})_1(z_{j})_{1+n_j}\cdots(z_{j})_{1+(n_i-1)n_j}=
    (z_{j})_1(z_{i})_1(z_{i})_{1+n_i}\cdots(z_{i})_{1+(n_j-1)n_i}$
    \end{enumerate}
\end{enumerate}

\end{thm}

\par We note here that in the case of $F(2,3)$, this presentation collapses even further, as the fact that $n_1=2$ means that all $(z_j)_i$ for $j>1$ can be expressed in terms of the other generators by using the relations in \ref{equivtreeonright}; however, this will not work in general.

\subsection{The Metric}
\par For Thompson's group $F(n)$, Burillo, Cleary and Stein have
shown in \cite{BurrilloFpMetric} that the metric is
quasi-isometric to the number of carets in the minimal $2$--ary
tree-pair diagram representative of the element of the group
(this also follows from Fordham's exact metric on $F(n)$ in \cite{length},
developed after \cite{BurrilloFpMetric}). This
however, is not the case for \Fnm{}. The simplest way to
illustrate this is to take powers of the generator $(y_j)_0$
for some $j\in\{2,...,k\}$ and to show that the number of
carets in the minimal tree-pair diagram representative for
$(y_j)_0^n$ grows exponentially as $n$ (and therefore
$|(y_j)_0^n|_{ X}$) increases linearly.

The convention in the literature on $F(n)$ is to talk about the
number of carets in the minimal tree-pair diagram
representative, but since in \Fnm\ the number of carets in an
\nary tree-pair diagram may not be the same in both trees, we
instead choose to refer to the number of leaves, as the number
of carets in an \nary\ tree-pair diagram is quasi-isometric to
the number of leaves.  (Since the number of leaves in an
$n$--ary tree containing $C$ carets will be $(n-1)C+1$, the
number of leaves, $L$, in an \nary tree will be such that
$(n_1-1)C+1\le L\le(n_k-1)C+1$, which implies that
$\frac{1}{n_k}C\le L\le n_kC$.)

\begin{defn}[quasi-isometry]
A quasi-isometry distorts lengths by no more than a constant
factor.  More formally, a quasi-isometry is a map between two
length functions $|x|_X,|x|_Y$ such that for all $x$ in the
group there exists a fixed $c$ such that
\[\frac{1}{c}|x|_X\le |x|_Y\le c|x|_X\]
\end{defn}

\begin{thm}\label{metricnotcaretnum}
The metric on \Fnm{} is not quasi-isometric to the number of
carets/leaves in the minimal tree-pair diagram representatives
of elements of \Fnm{}.
\end{thm}
\begin{proof}
We prove this by showing that the number of leaves in the
minimal tree-pair diagram representative for $(y_j)_0^n$ grows
exponentially as $|(y_j)_0^n|_{ X}=n$ increases linearly, for
$n\in\mathbb{N}$.  We let $(y_j)_0=(T_-,T_+)$ and
$(y_j)_0^n=(T_-^n,T_+^n)$ denote minimal representatives.

Now we prove this by induction.  Our inductive hypothesis is
that $T_-^n$ is an $n_1$--ary tree and that $T_+^n$ is a
balanced $n_j$--ary tree with $n_j^n$ leaves.  We can see that
this is the case when $n=1$ simply by looking at the minimal
tree-pair diagram representative of $(y_j)_0$ given in Figure
\ref{infgenFone}.  Now we suppose that these conditions hold
for some $n\ge1$.  Then when computing the product
$(y_j)_0^{n-1}(y_j)_0$, we must add carets to \Ttree\ and
$(T_-^{n-1},T_+^{n-1})$ to obtain $(T_-^*,T_+^*)$ and
$\left((T_-^{n-1})^*,(T_+^{n-1})^*\right)$ so that
$T_+^*\equiv(T_-^{n-1})^*$.  But since $T_+$ consists of one
$n_j$--ary caret and $T_-^{n-1}$ is an $n_1$--ary tree, we must
add one $n_j$--ary caret to each leaf in $T_-^{n-1}$ and by
extension to each leaf in $T_+^{n-1}$; these are the only
carets we need to add to $(T_-^{n-1},T_+^{n-1})$, so
$(T_+^{n-1})^*$ is a balanced $n_j$--ary tree with
$n_j^{n-1}n_j=n_j^n$ leaves.  Similarly, we must add only
$n_1$--ary carets to $T_+$ and by extension to $T_-$, so
$T_-^*$ is an $n_1$--ary tree.  So we have
$(T_-^{n},T_+^{n})\equiv(T_-^*,(T_+^{n-1})^*)$, and by
Corollary \ref{twocarettypes}, $(T_-^*,(T_+^{n-1})^*)$ is
minimal.
\end{proof}

We now generalize this idea to produce a lower bound for the
metric on \Fnm{} in terms of the number of leaves in an element's minimal tree-pair diagram representative.  We introduce the following lemma, which will be necessary for our proof of the lower bound.

\begin{lem}\label{leavesaddedupperbd}
Suppose there exists fixed $a\in\mathbb{N}$ and $x\in\Fnm$ such
that $\forall g\in X\cup X^{-}$, $L(xg)-L(x)\le aL(x)$. Then
for $g_1,...,g_n\in X\cup X^{-1}$,
\[L(xg_1\cdots g_n)\le(1+a)^nL(x)\]
\end{lem}
\begin{proof}
We prove this by induction.  For $n=1$ we have
\[L(xg_1)=L(x)+L(xg_1)-L(x)=L(x)+aL(x)=(a+1)L(x)\]
so our inductive hypothesis holds for $n=1$. Now we suppose
that it holds for all values up to arbitrary $n-1$: Then
\begin{eqnarray*}
    L(xg_1\cdots g_n) & \le & L(xg_1\cdots g_{n-1})+a(L(xg_1\cdots g_{n-1}))\\
      & \le &
     (1+a)^{n-1}L(x)+a((1+a)^{n-1}L(x))
      =  (1+a)^nL(x)
\end{eqnarray*}
\end{proof}

\begin{defn}[depth, $D(T)$, $D(x)$, and level]
The depth of a tree $T$ is the maximum distance from the root vertex to any leaf vertex, and the depth of an element $x$ is the maximal depth of the two trees in its minimal tree-pair diagram representative.  We use $D(T)$ and $D(x)$ to denote these depths, respectively.  A level is the subgraph of carets in a tree which are the same distance from the root vertex.
\end{defn}

\begin{thm}[Metric Lower Bound]\label{lowerbd}
For all elements $x\in \Fnm$, there exists fixed $B\in\mathbb{N}$ such that
\[|x|_{ X}\ge \log_BL(x)\]
\end{thm}
\begin{proof}
We consider multiplication of an arbitrary non-trivial element
$x=\Ttree\in \Fnm$ by $g=\Stree\in X\cup X^{-1}$.  We define
$\langle v_1,...,v_k\rangle$ so that $v_i$ is the maximum value
of the $n_i$--ary valence of all leaves in $S_+$. In order to
compute the product $xg$, we must make $T_-$ equivalent to
$S_+$, so we may need to add carets to $T_-$, and to $T_+$ by
extension, in order to achieve this.  The number of carets we
need to add will be bounded from above by the number of carets
needed to increase the valence of each leaf in $T_-$ by
$\langle v_1,...,v_k\rangle$; this is equivalent to adding a
balanced tree to each leaf in $T_-$with valence $\langle
v_1,...,v_k\rangle$.

\par But if we let $b=\frac{n_k-1}{n_1-1}$, $\forall g\in X\cup X^{-1}$,
$D(g)\le b+1$, and so the maximum $n_1$--ary
valence is $b+1$.  Additionally, there exists at most a
single $j\in\{2,...,k\}$ such that the $n_j$--ary valence is 1,
and the $n_l$--ary valence for all $l\ne1,l\ne j$ is 0.  So the
balanced tree to be added to each leaf of $T_-$ will have at most one level consisting of $n_j$--ary
carets and $b+1$ levels consisting of
$n_1$--ary carets, which implies that this balanced tree will
have at most $n_kn_1^b$ leaves.  So adding
this balanced tree to every leaf in \Ttree\ will add
$\left(n_kn_1^b-1\right)L(x)$
leaves to \Ttree\ so that
\[L(xg)\le n_kn_1^bL(x)\]
 And so, by
Lemma \ref{leavesaddedupperbd}, this implies that for arbitrary
$g_1,...,g_n\in X\cup X^{-1}$
\[L(xg_1\cdots g_n)\le\left(n_kn_1^b\right)^nL(x)\le a\left(n_kn_1^b\right)^n\]
where we obtain the second inequality by replacing $x$ with $g$; this inequality holds because
$L(x)\le a-1$ $\forall g\in X\cup X^{-1}$ for $a=n_k+n_1$.
So for
any $x$ in \Fnm{}, where $|x|_{ X}=n+1$ and
$B=n_kn_1^b$
\[L(x)\le aB^n\]

Taking the log of both sides and rewriting, and noting that $0<\log_B a<1$ yields:
\[|x|_{ X} > \log_B L(x)\]
\end{proof}

\begin{rmk}\label{sharplow}
The order of the lower bound given in Theorem \ref{lowerbd} is
sharp.
\end{rmk}
\begin{proof}
This follows immediately from the proof of Theorem
\ref{metricnotcaretnum}:
\[L\bigl({(y_j)_0^n}\bigr)=n_j^{n}\]
and therefore
\[|(y_j)_0^n|_{ X}\le n= \log_{n_j}\Bigl(L\bigl((y_j)_0^n\bigr)\Bigr)\]
\end{proof}

\par Now we proceed to find an upper bound for the metric.  Our proof of the upper bound of the metric will use the
normal from which was developed in the previous section.  We
recall that the normal form of a positive word $w$ will be of
the following form (see Theorem \ref{NF} for details):

\[(y_{\delta_1})_{\beta_1}\cdots(y_{\delta_N})_{\beta_N}(z_{\epsilon_1})_{\alpha_1}^{e_1}\cdots(z_{\epsilon_M})_{\alpha_M}^{e_M}\]

\par We prove results for positive words, and
then a simple corollary is all that is needed to extend this to
all words in \Fnm.

\begin{thm}[Metric Upper Bound]\label{lengthfingen}
For any positive word $w$ in \Fnm{}, there exist fixed
$c\in\mathbb{N}$ such that:
\[|w|_{X}<cL(w)\]
\end{thm}

\begin{proof}
To begin, we prove our results for the standard infinite
generating set $Z$; then we will extend them to the finite
generating set $X$.  First we will show that for a positive
word $w$
\[|w|_{Z}\le L(w)\]
We begin by noting that the total number of $(z_j)_i\in Z$
generators present in the normal form expression for a positive
word in \Fnm{} is equal to the number of non-right carets in
the range tree of the minimal tree-pair diagram representative;
this quantity can be expressed by the sum $\sum_{i=1}^M e_i$.
We can see this by considering how each generator in the normal
form is derived from the minimal tree-pair diagram (see Theorem
\ref{NF}).

Next we note that if we let $r(w)$ denote the number of carets on the right edge of the range tree of the minimal tree-pair diagram representative of $w$, then for a positive word $w$ ,
\[N\le r(w)\]
where $N$ is taken from the normal form expression for $w$ (see
Theorem \ref{NF}).  To see this we need only recall that each
$(y_j)_i$ generator in $NF(w)$ will contribute exactly one
caret to the right side of the range tree in the minimal
tree-pair diagram of $w$.

So if we let $N(w)$ denote the number of carets in the range tree of
the minimal tree-pair diagram representative of $w$, we have
\[N(w)\ge\displaystyle\sum_{i=1}^M e_i+N\]
where $e_i$, $N$ and $M$ are taken from $NF(w)$ (see Theorem \ref{NF}).
But since $\displaystyle\sum_{i=1}^M e_i+N$ is just the number of
generators present in $NF(w)$, we must have
\[|w|_Z\le N(w)\le L(w)\].

\par Now we are ready to use our results for $|x|_Z$ to derive an upper bound for $|x|_X$.
Let $R(i)$ be the remainder of $\frac{i}{n_1-1}$.  Then by
looking at the relator
$(\gamma_l)_j(z_1)_i=(z_1)_i(\gamma_l)_{j+n_1-1}$
when $i<j$, where $\gamma\in\{y,z\}$, it is clear that
\[(z_j)_i=(z_1)_0^{-i/(n_1-1)}(z_j)_{R(i)}(z_1)_0^{i/(n_1-1)}\]
so that $(z_j)_{R(i)}\in X$.  Then by substituting these relator types into
$$NF(w)  =  (y_{\delta_1})_{\beta_1}\cdots(y_{\delta_N})_{\beta_N}(z_{\epsilon_1})_{\alpha_1}^{e_1}\cdots(z_{\epsilon_M})_{\alpha_M}^{e_M}$$
we obtain:
\begin{eqnarray*}
|w|_X & \le & |w|_Z+\frac{\beta_1}{n_1-1}+\frac{\alpha_M}{n_1-1}+\frac{|\beta_N-\alpha_1|}{n_1-1}+\displaystyle\sum_{i=1}^{N-1}\left|\frac{\beta_i-\beta_{i+1}}{n_1-1}\right|+\displaystyle\sum_{i=1}^{M-1}\left|\frac{\alpha_i-\alpha_{i+1}}{n_1-1}\right|\\
     & \le & |w|_Z+\frac{4\max\{\beta_N,\alpha_{M}\}}{n_1-1}\le L(w)+\frac{4L(w)}{n_1-1}\le5L(w)
\end{eqnarray*}
since $\beta_N$ and $\alpha_M$ both denote leaf index numbers
in the minimal tree-pair diagram representative of $w$ and
therefore $L(w)>\beta_N,\alpha_M$.
\end{proof}

\par Now we extend the results for
positive words to all words in \Fnm{}.

\begin{cor}\label{upperbd}
$\exists$ fixed $d\in\mathbb{N}$ such that for any (not
necessarily positive) $w\in \Fnm{}$
\[|w|_X\le dL(w)\]
\end{cor}
\begin{proof}
For any word $w$, we can factor it uniquely into the product
$w=w_+w_-^{-1}$ where $w_+$ and $w_-$ are both positive words.
Clearly letting $d=2c$ yields
\[|w|\le|w_+|+|w_-|  \le  cL(w_+)+cL(w_-)\le dL(w)\]
\end{proof}

\begin{rmk}\label{sharpup}
The order of the upper bound given in Corollary \ref{upperbd}
is sharp.
\end{rmk}
\begin{proof}
We will show that an arbitrary positive element $w=\Ttree$ has
length which is quasi-isometric to $L(w)$.  First we relate the
depth of an element to the length, and then we use this to
bound the length with respect to $L(w)$.  For any word
$w\in\Fnm$, there exists fixed $c\in\mathbb{N}$ such that
\[D(w)\le c|w|_X\]
where we recall that $D(w)$ denotes the depth of the minimal
tree-pair diagram of $w$.

We can prove this by induction.  Let $c=\frac{n_k-1}{n_1-1}+2$; our inductive hypothesis will
be
\[D(g_1\cdots g_n)\le cn\]
for $g_1,...,g_n\in X\cup X^{-1}$ for some $n\in\mathbb{N}$.
If $g_1,...,g_n$ is a minimal length expression for $w$, and
this inductive hypothesis holds, then it is clear that our
lemma holds.  By looking at the minimal tree-pair diagram
representatives of all generators in $X$, we can see that
$\forall g_i\in X\cup X^{-1}$
\[D(g_i)\le\frac{n_1-1}{n_1-1}+2<c\]
From this fact and the proof of Theorem \ref{lowerbd}, we can
see that the maximum number of levels added when multiplying by
a generator or its inverse is $c$, so by our inductive
hypothesis, $\forall g_1,...,g_{n+1}\in X\cup X^{-1}$:
\[D(g_1,...,g_{n+1})  \le  D(g_1,...,g_n)+c\le c(n+1)\]

Now we consider a positive element $w=\Ttree$ in \Fnm.  Since $T_-$
consists of a string of $n_1$--ary carets which are all on the
right side of the tree,
\[L(T_-)=D(T_-)(n_1-1)+1\]
And it is clear that $D(w)=D(T_-)$.  So we have
\[\frac{L(w)-1}{n_1-1}=D(w)\le c|w|_X\]
So for $d=3c(n_1-1)$ we have:
\[|w|_X  \ge  \frac{1}{c}\cdot\frac{L(w)-1}{n_1-1}\ge \frac{2}{d}L(w)\]
since $L(w)-1\ge\frac{2}{3}L(w)$, as all nontrivial $w$ have
$L(w)\ge3$.
\end{proof}

\bibliography{Fnm}
%%\nocite{*}
\bibliographystyle{amsplain}

\end{document}